\documentclass[11pt]{amsart} 
 
\usepackage{amsmath, amsthm, amssymb}
\usepackage{graphics}
\usepackage{pinlabel}
\usepackage{color}
\usepackage[colorlinks,citecolor=blue,linkcolor=red]{hyperref}
\usepackage{graphicx}
\usepackage{nicefrac}						
\usepackage{mathtools}
\usepackage{float}
\usepackage{tikz-cd}

\usepackage{enumerate}

\usepackage{fullpage}

\usepackage{secdot}
\usepackage{caption}

\tikzcdset{every label/.append style = {font = \large}}

\DeclareMathOperator{\Hom}{H^{sc}_1}
 
\newtheorem{theorem}{Theorem}[section]
\newtheorem{lemma}[theorem]{Lemma}
\newtheorem{remark}[theorem]{Remark}

\newtheorem{corollary}[theorem]{Corollary}

\newtheorem{conjecture}[theorem]{Conjecture}
\newtheorem{proposition}[theorem]{Proposition}
\newtheorem{question}{Question}

\newtheorem*{mainthm*}{Main Theorem}

\everymath{\displaystyle}

\title{Characterizing covers via simple closed curves}
\author{Tarik Aougab, Max Lahn, Marissa Loving, Yang Xiao}

\begin{document}

\maketitle 

\begin{abstract} Given two finite covers $p: X \to S$ and $q: Y \to S$ of a connected, oriented, closed surface $S$ of genus at least $2$, we attempt to characterize the equivalence of $p$ and $q$ in terms of which curves lift to simple curves. Using Teichm\"uller theory and the complex of curves, we show that two regular covers $p$ and $q$ are equivalent if for any closed curve $\gamma \subset S$, $\gamma$ lifts to a simple closed curve on $X$ if and only if it does to $Y$. When the covers are abelian, we also give a characterization of equivalence in terms of which powers of simple closed curves lift to closed curves. \end{abstract}

\section{Introduction}

It is a corollary of a renowned theorem of Scott \cite{Scott} that every closed curve on a hyperbolic surface $S$ lifts to a simple closed curve on some finite cover of $S$. This result was made effective by Patel \cite{Patel}, and more work has been done since then to improve the bound on the degree of the required cover, as well as to study the connection between this degree and the self intersection number of the curve (see Gupta--Kapovich \cite{Gupta}, Gaster \cite{Gaster}, and Aougab--Gaster--Patel--Sapir \cite{GAPS}). In the spirit of this work, it is natural to ask the following question, which motivates the main result of this paper.

\begin{question} \label{simple-lifting}
What information about $p: X \to S$ and $q: Y \to S$ can be derived from understanding how curves on $S$ \emph{lift simply} to $X$ and $Y$?  
\end{question}

Our first result addresses this question in the setting of regular finite covers of a closed surface $S$ with genus at least $2$. In particular, we characterize when two regular finite covers are equivalent in terms of which closed curves have simple elevations, where an {\it elevation} of a closed curve $\alpha \subset S$ along a covering map $f: \Sigma \to S$ is a closed curve on $\Sigma$ that projects to $\alpha$ under $f$. 

\begin{theorem} \label{thm:regular} If $p:X \to S$ and $q: Y \to S$ are two regular covers so that given any closed curve $\gamma \subset S$, there exists simple elevations of $\gamma$ to $X$ if and only if there exists simple elevations of $\gamma$ to $Y$, then $p$ and $q$ are equivalent covers. \end{theorem}

We prove a generalization of Theorem~\ref{thm:regular} to non-simple curves, which is stated precisely at the end of Section \ref{motivation} (see Theorem~\ref{thm:nonsimple}). Another variation of Question \ref{simple-lifting} would be to consider lifts of simple curves rather than curves that lift simply. This leads us to ask the following.

\begin{question} \label{lifts-of-simple}
What information about $p:X \to S$ and $q:Y \to S$ can be derived from understanding how \emph{simple} curves on $S$ lift to $X$ and $Y$? \end{question}

Our second result answers this question in the setting of abelian covers. Given a cover $f: \Sigma \to S$ and a simple closed curve $ \gamma \subset S $, we let $n_f(\gamma)$ denote the minimum positive integer $k$ such that $\gamma$ admits an elevation of degree $k$ along $f$. 

\begin{theorem} \label{thm:abelian} Let $p: X\to S$ and $q: Y \to S$ be finite-degree abelian covers of $S$. If $n_p(\gamma) = n_q(\gamma)$ for all simple closed curves $\gamma \subset S$, then $p$ and $q$ are equivalent covers. \end{theorem}

\subsection{Motivation, Sunada's construction, and non-simple curves} 
\label{motivation}

The authors arrived at Theorems \ref{thm:regular} and \ref{thm:abelian} while studying \textit{isospectral} hyperbolic surfaces: hyperbolic surfaces that have the same unmarked length spectra but which are not necessarily isometric. Almost all known examples of isospectral pairs come from a construction (or some variant thereof) due to Sunada \cite{Sunada} which we briefly summarize here. See \cite{Buser} for a more in depth introduction. 

Let $G$ be a finite group containing a pair of \textit{almost conjugate} subgroups $H, K$, meaning that the cardinality of the intersection of $ H $ with any conjugacy class $ [ g ] $ in $ G $ agrees with that of $ K $:
\[ |H \cap [g]| = |K \cap [g]|. \]
Let $S$ be an orientable surface of finite type so that $\pi_1(S)$ admits a surjective homomorphism $\rho: \pi_1(S) \rightarrow G$. Then if one equips $S$ with a hyperbolic metric, the pullbacks of that metric to the covers corresponding to the subgroups $\rho^{-1}(K), \rho^{-1}(H) < \pi_1(S)$ will be isospectral. So long as $H,K$ are not conjugate in $G$, for a generic choice of the initial metric, these hyperbolic surfaces will not be isometric. 

An interesting open question is whether or not there exist non-isometric surfaces which are \textit{simple length isospectral}, meaning that the multi-set of lengths corresponding only to simple closed geodesics coincide. Sunada's construction provides a natural way to test this question. Indeed, letting $S_H$ and $S_K$ denote the isospectral covers associated to $\rho^{-1}(H)$ and $ \rho^{-1}(K)$ as above, Sunada's construction yields a length-preserving bijection $\phi$ between the sets of closed geodesics on $S_H$ and $S_K$. If for any such $S_H $ and $ S_K$ where $H$ and $K$ are not conjugate, $\phi$ happens to send the lengths corresponding to \textit{simple} closed geodesics on one surface to the simple closed geodesics on the other, one immediately obtains an answer to the question: such an example would imply that simple length isospectrality need not imply an isometry.

Along these lines, Maungchang \cite{Maungchang} explored the example $ G= (\mathbb{Z}/8\mathbb{Z})^{\times} \ltimes \mathbb{Z}/8\mathbb{Z} $, where the semi-direct product is with respect to the standard action of the group of multiplicative units, and where 
\begin{align*}
    H & = \{(1,0), (3,0), (5,0), (7,0)\}, & K &= \{(1,0), (3,4), (5,4), (7,0)\}
\end{align*}
Letting $S$ be the closed surface of genus $2$ and $ \pi_1(S) = \langle a,b,c,d \mid [a,b][c,d] =1 \rangle $ be the standard presentation of its fundamental group, one obtains the homomorphism $\rho: \pi_1(S) \rightarrow G$ defined by
\begin{align*}
\rho ( a ) & = ( 3 , 0 ) & \rho ( b ) & = ( 5 , 0 ) & \rho ( c ) & = ( 1 , 0 ) & \rho ( d ) & = ( 1 , 1 )
\end{align*}
Then Maungchang directly demonstrates that for a generic choice of initial metric on the genus $2$ surface, the hyperbolic surfaces $S_H, S_K$ are \textit{not} simple length isospectral. His strategy is as follows: 
\begin{enumerate}
\item Exhibit a closed curve $\gamma$ on $S$ that admits a different number of simple elevations to $S_H$ than it does to $S_K$. The length-preserving bijection $\phi$ coming from Sunada's construction has the property that for any closed curve on $S$, $\phi$ relates its set of elevations on $S_H$ to those on $S_K$. Thus, the existence of $\gamma$ implies that if $S_H$ and $S_K$ are simple length isospectral, then any length-preserving bijection between the sets of simple closed geodesics is not induced by $\phi$. 

\item A bound on the degree of the covering spaces precludes all but finitely many curves on $S$ from having simple elevations that could possibly make up for the discrepancy coming from simple elevations of $\gamma$ established by (1). By checking these finitely many possibilities and verifying for a fixed hyperbolic metric that the lengths do not make up for this discrepancy, Maungchang obtains the desired result by appealing to real-analyticity of length functions over Teichm{\"u}ller space (if lengths disagree on one metric, they disagree for almost all metrics).
\end{enumerate}

The authors conjecture that simple length isospectrality implies isometry for hyperbolic surfaces:

\begin{conjecture} \label{conj:simple}
Hyperbolic surfaces with the same multi-set of simple lengths are isometric.
\end{conjecture}

A reasonable approach to Conjecture \ref{conj:simple} would be to first establish it for pairs of isospectral surfaces arising from Sunada's construction, since these pairs are natural candidates for counterexamples. One could hope that Maungchang's strategy outlined above is generalizable to any pair of isospectral surfaces arising from Sunada's construction. 

A first step towards achieving such a generalization would be to address step $(1)$ above, and to show that if two covers of a surface are not equivalent, there is some closed curve $\gamma$ on the base surface admitting a different number of simple elevations to the two covers. In the setting where the covers are regular, if one elevation is simple, all must be. This motivates Theorem \ref{thm:regular}, since ``admitting a different number of simple elevations to the two covers" reduces to ``admitting a simple elevation to one cover but not the other". 

Unfortunately, the covering spaces that arise in Sunada's construction are necessarily never regular, as almost conjugate subgroups must be conjugate (and in fact, equal) when one is normal. The authors conjecture that the assumption of regularity can be dropped in Theorem \ref{thm:regular}. 

\begin{conjecture} \label{conj:notregular}
Theorem \ref{thm:regular} holds without the assumption that the covers are regular. 
\end{conjecture}

While it is not yet clear how our arguments can be promoted to prove Conjecture \ref{conj:notregular}, we note that they can be used to show the following generalization of Theorem \ref{thm:regular}:  

\begin{theorem} \label{thm:nonsimple} Let $N, M \in \mathbb{N}$ and let $p:X \rightarrow S, q:Y \rightarrow S$ be regular finite covers of a closed orientable surface $S$ with genus at least $2$. Suppose that for any essential closed curve $\gamma \subset S$, $\gamma$ admits an elevation to $X$ with at most $N$ self intersections if and only if it admits an elevation to $Y$ with at most $M$ self intersections. Then $p$ and $q$ are equivalent covers. 
\end{theorem}

\subsection*{Acknowledgment} Aougab is supported by NSF grant DMS 1807319. Loving is supported by an NSF Postdoctoral Research Fellowship under Grant No. DMS 1902729. 

\section{Preliminaries} 

We devote this section to recalling concepts and results that are used in the proof of Theorem~\ref{thm:regular}. 

\subsection{Hyperbolic spaces and Gromov boundaries} \label{hyp}
Let $(X,d)$ be a (not necessarily proper) geodesic metric space. Then for some $\delta>0$, $X$ is said to be $\delta$-\textit{hyperbolic} if for any geodesic triangle in $X$, the $\delta$-neighborhood of the union of any two sides contains the third side. 

Any $\delta$-hyperbolic space admits a \textit{Gromov boundary}, $\partial_{\infty}X$, generalizing the boundary sphere at infinity of hyperbolic space. To define $\partial_{\infty}X$, one first fixes a basepoint $p \in X$; then given two points $x,y \in X $, their \textit{Gromov product} with respect to $p$ is
\[
( x , y )_{ p } = \frac{ d ( x , p ) + d ( y , p ) - d ( x , y ) }{ 2 } .
\]
A sequence of points ${x_i} \subset X$ is said to be \textit{admissible} if $ \lim_{ i , j \to \infty }{ ( x_{ i } , x_{ j } )_{ p } } = \infty  $, and two admissible sequences ${x_i}, {y_i}$ are said to be equivalent if $\lim_{i \rightarrow \infty}(x_i, y_i)_p = \infty.$ 

Then as a set, $\partial_{\infty}X$ is the set of equivalence classes of admissible sequences. Given $K>0$ and $b \in \partial_{\infty}X$, define 
\[
U_{ K } ( b ) = \left \{ \left[ \left( x_{ i } \right) \right] \in \partial_{ \infty } X \:\middle |\: \liminf_{ i , j \to \infty }{ \left( x_{ i } , y_{ j } \right )_{ p } } \geq K \textrm{ for some } \left( y_{ j } \right ) \in b \right \} .
\]

These subsets form a basis for a topology on $\partial_{\infty}X$ with respect to which isometries of $X$ induce homeomorphisms. Moreover, a quasi-isometric embedding $\varphi: X \to Y$ between two $\delta$-hyperbolic spaces induces a continuous embedding $\widehat{\varphi}: \partial X \to \partial Y$. We briefly describe the map $\widehat{\varphi}$. 

Given a point $z \in \partial X$, choose a geodesic ray $\gamma \subset X$ whose endpoint on $\partial X$ is $z$ (which is to say that any unbounded sequence of points along $\gamma$ is in the class of $z$). Then the image of $\gamma$ under $\varphi$ is a quasi-geodesic ray $\gamma'$ in $Y,$ which implies that any unbounded sequence of points along $\gamma'$ is admissible and that therefore $\gamma'$ is naturally associated to some point in $\partial Y$ which we call $z'$. The map $\widehat \varphi$ sends $z$ to $z'$. It is straightforward to show that $\widehat \varphi$ is well-defined.

\subsection{Teichm{\"u}ller space and geodesic laminations} 

Let $S$ be a connected surface with negative Euler characteristic. The \textit{Teichm\"uller space} of $S$, denoted $\mathcal T(S)$, can be defined as the space of equivalence classes of marked hyperbolic structures on $S$, where two hyperbolic structures $f:S \to X$ and $g:S \to Y$ are said to be equivalent if $f \circ g^{-1}$ is isotopic to an isometry. An equivalent characterization of $\mathcal T(S)$ is as the space of $\text{PGL}(2, \mathbb R)$ conjugacy classes of discrete faithful representations of $\pi_1(S)$ into $\text{PSL}(2, \mathbb R)$. We will make use of both of these perspectives throughout Section \ref{section:regular}.

Abusing notation slightly, let $S$ now be a complete hyperbolic surface of finite type and without boundary. A \textit{geodesic lamination} of $S$ is a compact subset of $S$ consisting of a disjoint union of simple geodesics. A geodesic lamination $\lambda$ is said to be \textit{minimal} if every leaf is dense in $\lambda$, and \textit{filling} if it intersects every simple closed geodesic. Laminations with both of these properties are called \textit{ending laminations}.

A \textit{measured geodesic lamination} is a geodesic lamination $\lambda$ equipped with a positive Borel measure $\mu$ on the set of arcs transverse to $\lambda$, which is invariant under transverse homotopy. The set of all measured geodesic laminations, denoted $\mathcal{ML}(S)$, admits a natural topology for which two points are close together if they induce approximately the same measure on sets of finitely many transverse arcs. This topology does not depend on the initial choice of hyperbolic metric and thus we can associate $\mathcal{ML}(S)$  to the underlying topological surface. 

Two measured geodesic laminations are \textit{projectively equivalent} if they have the same underlying geodesic lamination and the transverse measures differ only by a scaling. Thurston showed that the space of projective measured laminations, denoted $\mathcal{PML}(S)$ (equipped with the quotient topology) yields a compactification of $\mathcal{T}(S)$ and is homeomorphic to a sphere of dimension $\mbox{dim}(\mathcal{T}(S))-1$. By $\mathcal{PMEL}(S)$, we will mean the subset of $\mathcal{PML}(S)$ consisting of points whose underlying lamination is ending, and we will denote by $\mathcal{EL}(S)$ the image of $\mathcal{PMEL}(S)$ under the quotient map sending a projective measured lamination to its underlying geodesic lamination. A geodesic lamination $\lambda$ is called \textit{uniquely ergodic} if there is only one projective equivalence class of transverse measures it supports.

We conclude this subsection with several useful results about $\mathcal{T}(S), \mathcal{ML}(S)$, and their behaviors under covering maps between surfaces. The first such result can be applied to show that the set of uniquely ergodic ending laminations are dense in the space of measured laminations.

\begin{theorem}[Lindenstrauss-Mirzakhani, Theorem $1.2$ in \cite{LM}] \label{LM-thick}

A measured lamination $\lambda\in \mathcal{ML}(S)$ has a dense orbit in $\mathcal{ML}(S)$ if and only if its support does not contain any simple closed curves.  
\end{theorem}

The next result states that covering maps induce isometric embeddings between Teichm\"uller spaces. A proof of Theorem \ref{thm:folklore} can be found on page 2153 of \cite{RafiSchleimer-covers}. 

\begin{theorem}[Folklore] \label{thm:folklore}

A finite covering map $p: X \to S$ induces an isometric embedding $\widetilde{p}: \mathcal{T}(S) \to \mathcal T(X)$. 

\end{theorem}

Building on Theorem \ref{thm:folklore}, the following result of Biswas--Mj--Nag allows for the extension of any map between Teichm\"uller spaces induced by a finite cover to a map between spaces of projective measured laminations, interpreted as boundaries at infinity $\partial \mathcal{T}(S)$ of the corresponding Teichm\"uller spaces. In particular, given a finite covering map $p: X \to S$ there is a natural map $\partial \mathcal{T}(S) \to \partial \mathcal T(X)$. 

\begin{theorem}[Biswas--Mj--Nag, \cite{Mj}] \label{thm:Mj}

A finite covering map $p: X \to S$ between closed, oriented hyperbolic surfaces induces a natural continuous injection $\widetilde{p}: \overline{\mathcal{T}(S)} \to \overline{\mathcal{T}(X)}$ between the corresponding Thurston compactified Teichm\"uller spaces. Furthermore, this map is the continuous extension of the holomorphic embedding from $\mathcal T(S) \to \mathcal T(X)$ induced by $p$. 
\end{theorem}

A consequence of Theorem \ref{thm:Mj} is that $\widetilde p$ is defined on $\partial \mathcal{T}(S)$ as follows: given a measured geodesic lamination $\lambda$ on $S$ its image under $\widetilde p$ is simply the measured geodesic lamination $p^{-1}(\lambda)$ on $X$ obtained as the inverse image of $\lambda$ under $p$.

There is a Finsler metric on $\mathcal{T}(S)$, called the \textit{Teichm\"uller metric}, in which the distance between two points is the logarithm of the infimal dilatation of quasiconformal homeomorphisms from one marked surface to the other, taken over all such homeomorphisms isotopic to the identity. A result of Masur demonstrates the abundance of geodesic rays in the Teichm\"uller metric with endpoints on $\partial \mathcal{T}(S)$.

\begin{theorem}[Masur, \cite{Masur1982}] \label{MasurBoundary} At every point $x \in \mathcal T(S)$ and in almost every direction, a Teichm\"uller geodesic ray based at $x$ has a limit on the Thurston boundary of $\mathcal T(S)$. 
\end{theorem}

\subsection{The Complex of Curves}

Given an orientable surface $S$ with negative Euler characteristic, the \textit{curve complex}, $\mathcal C(S)$, is a flag simplicial complex whose vertices correspond to isotopy classes of essential simple closed curves and whose edges represent pairs of such classes that can be realized disjointly on $S$. By identifying each simplex with a standard simplex with unit length edges in the appropriate Euclidean space, $\mathcal{C}(S)$ becomes a metric space. 

A germinal result of Masur--Minsky \cite{MasurMinskyHyperbolic} states that the curve complex is $\delta$-hyperbolic. It follows that $\mathcal{C}(S)$ admits a Gromov boundary, which was characterized by Klarreich, as follows. Define $\mbox{sys}: \mathcal{T}(S) \rightarrow \mathcal{C}(S)$ by sending $X \in \mathcal{T}(S)$ to its \textit{systole}, the simple closed curve admitting the shortest geodesic representative on $X$. 

\begin{remark} \label{coarse} Note that $\mbox{sys}$ is technically not well-defined since the systole need not be unique. However, if two simple closed geodesics are simultaneously shortest, they can intersect at most once. It follows that the set of systoles for any $X \in \mathcal{T}(S)$ represents a subset of $\mathcal{C}(S)$ with diameter at most $2$. Thus, $\mbox{sys}$ is said to be \emph{coarsely well-defined}.
\end{remark}

Klarreich shows that $\mbox{sys}$ induces a map between $\mathcal{PMEL}(S)$ and $\partial_{\infty}\mathcal{C}(S)$ which allows for a characterization of the latter in terms of the former:

\begin{theorem}[Klarreich, \cite{Klarreich}]
\label{thm:ending-laminations}
The map $\mbox{sys}$ extends to a continuous map 
\[ \mbox{sys}_{\ast}: \mathcal{PMEL}(S) \rightarrow \partial_{\infty}\mathcal{C}(S), \]
factoring through the quotient map between $\mathcal{PMEL}(S)$ and $\mathcal{EL}(S)$, and inducing a homeomorphism between $\mathcal{EL}(S)$ and $\partial_{\infty}\mathcal{C}(S)$. 

\end{theorem}

We will need the following two important results regarding the coarse geometry of the curve complex. The first result gives a relation between the curve complex of a surface $S$ and the curve complex of a regular cover $\Sigma$ of $S$.

\begin{theorem}[Rafi--Schleimer, \cite{RafiSchleimer-covers}] \label{thm:RS-curvecomplex}
Let $P: \Sigma \to S$ be a covering map and $\Pi: \mathcal C(S) \to \mathcal C(\Sigma)$ be the covering relation where $b \in \mathcal C(S)$ is related to $\beta \in \mathcal C(\Sigma)$, if $P(\beta) = b$. The map $\Pi: \mathcal C(S) \to \mathcal C(\Sigma)$ is a $Q$-quasi-isometric embedding, with $Q$ depending only on the topology of $ S $ and the degree of $ P $. 
\end{theorem} 

Note that $\Pi$ is not well-defined since the pre-image of a simple closed curve may very well be a multi-curve, but it is coarsely well-defined.

The second result establishes the relationship between the topology of two surfaces based on the coarse geometry of their respective curve complexes.

\begin{theorem}[Rafi-Schleimer, \cite{RafiSchleimer-rigid}] \label{thm:RS-quasi}
Let $S$ and $\Sigma$ be orientable, connected, compact surfaces of genus at least $2$. A quasi-isometry $ \Psi \colon \mathcal C ( S ) \to \mathcal C ( \Sigma ) $ is induced by a homeomorphism $\psi: S \to \Sigma$.
\end{theorem}

\section{Regular Covers and Simple Curves}

\label{section:regular}

Throughout this section, fix finite regular covers $p: X \to S$ and $q: Y \to S$ of $S$. The only place in the proof of Theorem \ref{thm:regular} where we will invoke the regularity of $p$ and $q$ is in establishing the following lemma.

\begin{lemma} \label{lemma:qi-curve-cx}
Suppose that for any closed curve $\gamma \subset S$, there exists a simple elevation of $\gamma$ to $X$ if and only if there exists a simple elevation of $\gamma$ to $Y$. Then $ \mathcal{ C } ( X ) $ and $ \mathcal{ C } ( Y ) $ are quasi-isometric, and hence $ X $ and $ Y $ are homeomorphic.
\end{lemma}

\begin{proof} Let $W$ be the regular cover of $X, Y,$ and $S$ that corresponds to $\pi_1(X) \cap \pi_1(Y) \subset \pi_1(S)$. This gives a diamond of regular covers of $S$, which is shown on the left of Figure \ref{fig:diamond}. In particular, we have $\pi = p\circ p' = q \circ q'$. We will use $W$ to define a quasi-isometry from $\mathcal C (X)$ to $\mathcal C (Y)$.

Let $\alpha_{X} \subset X$ be an essential simple closed curve and let $A_W$ be the multi-curve consisting of all elevations of $\alpha_{X}$ to $W$. Now consider the image $ A_{ Y } := q' ( A_W ) $ of $A_W$ on $Y$ under the regular cover $q': W \to Y$ shown in Figure \ref{fig:diamond}. 

We claim that $ A_{ Y } $ is a union of simple closed curves on $Y$. To see this, let $\alpha_Y = q'(\alpha_W)$ for some simple closed curve $\alpha_W \subset A_W$. Let $\alpha_S \subset S$ denote the projection of $\alpha_W$ to $S$ under the regular cover $\pi: W \to S$. By the assumptions,
\[
p(\alpha_{X}) = p \left( p'(\alpha_W ) \right) =\pi(\alpha_W) = \alpha_{ S } = \pi ( \alpha_{W } ) = q \left ( q' ( \alpha_{ W } ) \right ) = q ( \alpha_{Y } ) .
\]
Hence $ \alpha_{ X } $ and $ \alpha_{ Y } $ are elevations of $ \alpha_{ S } $. Since $ \alpha_{ X } $ is simple, our assumptions imply that $ \alpha_{ S } $ has some simple elevation to $ Y $. Since $ q \colon Y \to S $ is regular, this implies that all elevations of $ \alpha_{ S } $ to $ Y $ are simple, and so $ \alpha_{ Y } $ is simple.\footnote{Notice our use of the regularity of $q$.} It follows that $ A_{ Y } $ is a union of simple closed curves, though not necessarily pairwise disjoint ones. 

We next claim that the curves in $A_Y$ constitute a bounded diameter subset of $\mathcal C(Y)$, with the bound depending only on the covering maps $p,q$ and not on $\alpha_{X}$. Consider the elevations $\Gamma$ of all curves in $A_{ Y }$ to $W$, which contains $A_{ W } $. Note that $\Gamma$ is the union of all multicurves $\widetilde{\gamma}$ that arise as the set of elevations of some $\gamma \in A_{ Y } $. Moreover, $\widetilde{\gamma}\cap \alpha_W \neq \emptyset$ for any $\gamma \in A_{ Y } $. Using this fact we argue that $\Gamma$ has diameter at most $3$ in $\mathcal{C} (W)$: given any two distinct simple closed curves $\beta_1, \beta_2 \in \Gamma$, they each belong to some multicurve $\widetilde{\gamma_i}$ coming from the elevations of $\gamma_i \in A_{ Y }$, where $i=1, 2$. If $\widetilde{\gamma_1}= \widetilde{\gamma_2}$, then $\beta_1$ and $\beta_2$ are distance $1$ apart. Otherwise, there exists $\alpha_i\in \widetilde{\gamma_i}\cap \alpha_W$ for $i=1,2$ such that $\{\beta_1, \alpha_1, \alpha_2, \beta_2\}$ is a path of length at most $3$ in $\mathcal{C} (W)$. 

By Theorem \ref{thm:RS-curvecomplex}, $\Pi:\mathcal C(Y) \to \mathcal C(W)$ is a quasi-isometric embedding. Since $\Pi(A_Y) = \Gamma$, the curves in $A_{ Y }$ give a bounded diameter subset of $\mathcal C(Y)$. Thus the assignment $\alpha_{X} \mapsto A_{ Y } $ gives a coarsely well-defined map $\Psi: \mathcal C(X) \to \mathcal C(Y)$. In fact, $\Psi$ is coarsely Lipschitz by the same argument: given two disjoint simple closed curves on $X$, the union of their pre-images will be a multi-curve on $W$.

By symmetry, there exists a coarsely well-defined Lipschitz map $\Phi: \mathcal C(Y) \to \mathcal C(X)$. We now show that $\Phi$ is a coarse inverse to $\Psi$. It is clear that $\Phi(\Psi(\alpha_{ X })) = \Phi(A_{ Y })$. Letting $A_{YW}$ denote the full pre-image of $A_Y$ under $q'$, we have that $ \Phi(\Psi(\alpha_{ X }))  = p'(A_{ Y W } ) $.

Since $\Psi(\alpha_{ X })= A_{ Y } $ has bounded diameter in $\mathcal{C}(Y)$ (bounded independently of $\alpha_{ X }$) and $\Phi$ is coarsely Lipschitz, then $p'(A_{ Y W })$ has bounded diameter (also independent of $\alpha_{ X }$) in $\mathcal{C}(X)$, and it contains $\alpha_{ X }$.  It follows that $\Phi \circ \Psi$ is quasi-isometric to the identity on $\mathcal{C}(X)$, and a completely analogous argument proves the same for $\Psi \circ \Phi$ on $\mathcal{C}(Y)$.  Thus, $\mathcal C(X)$ and $\mathcal C(Y)$ are quasi-isometric and by Theorem \ref{thm:RS-quasi}, $X$ and $Y$ are homeomorphic. \end{proof}

Note that Theorem \ref{thm:RS-curvecomplex} gives us quasi-isometric embeddings of $\mathcal C(S)$ into both $\mathcal C(X)$ and $\mathcal C(Y)$. Thus, we have two different quasi-isometric embeddings of $\mathcal C(S)$ into $\mathcal C(Y)$. The first comes directly from Theorem \ref{thm:RS-curvecomplex} and we will call it $q^{*}: \mathcal C(S) \to \mathcal C(Y)$. The second comes from the composition of the quasi-isometry $\mathcal C(X) \to \mathcal C(Y)$ constructed in Lemma \ref{lemma:qi-curve-cx} with the quasi-isometric embedding $\mathcal C(S) \to \mathcal C(X)$  from Theorem \ref{thm:RS-curvecomplex}, and we will call it $g^{*}: \mathcal C(S) \to \mathcal C(Y)$.

\begin{lemma}\label{BoundaryCurveComplex} The quasi-isometric embeddings $q^{*}, g^{*}: \mathcal C(S) \to \mathcal C(Y)$ induce the same map on $\partial \mathcal C(S) \to \partial \mathcal C(Y)$.\end{lemma}

\begin{proof} It suffices to show that $q^{*}$ and $g^{*}$ are within bounded distance of each other; the lemma will then follow by the basic properties of hyperbolicity and Gromov boundaries outlined in Section \ref{hyp}. 

Let $\alpha \in \mathcal C(S)$ and let $A_{X}, A_Y,$ and $A_W$ be the collections of elevations of $\alpha$ to $X$, $Y$, and $W$, respectively (note that $A_{Y}$ is defined slightly differently here than in Lemma \ref{lemma:qi-curve-cx}). We will show that $q^{*}(\alpha)$ and $g^{*}(\alpha)$ are contained in $A_Y$. Recall from Theorem \ref{thm:RS-curvecomplex} that $q^{*}(\alpha) \in A_Y$ by definition. Let $p^{*}:\mathcal C(S) \to \mathcal C(X)$ be the map induced by $p: X \to S$ given by Theorem \ref{thm:RS-curvecomplex}. Then $g^{*}:\mathcal C(S) \to \mathcal C(Y) $ is the composition of $p^{*}$ with the map $\Psi: \mathcal C(X) \to \mathcal C(Y)$ from the proof of Lemma \ref{lemma:qi-curve-cx}. 

Note that the elevations of $A_{X}$ (resp. $A_Y$) to $W$ under the cover $W \to X$ (resp. $W \to Y$) is exactly $A_W$. Since $\Psi: \mathcal C(X) \to \mathcal C(Y)$ factors through $\mathcal C(W)$ in its definition, we have $\Psi(A_{X}) = A_{ Y }$. It follows that $g^{*}(\alpha) = \Psi(p^{*}(\alpha))\in \Psi(A_{X}) = A_Y$. Since $A_Y$ is a diameter one subset of $\mathcal C(Y)$, the claim follows immediately. \end{proof}

\begin{figure} 
\begin{tikzcd}[column sep={6em,between origins}, row sep=huge]
& W \arrow [dl, "p'", swap] \arrow[dr, "q'"] \arrow[dd, "\pi"] & \\
X \arrow[dr, "p", swap] & & Y \arrow[dl, "q"] \\ 
& S & 
\end{tikzcd} \qquad
\begin{tikzcd}[column sep={6em,between origins}, row sep=huge]
& \mathcal C(S) \arrow[dl, "p^{*}", swap] \arrow[dr, "q^{*}"] & \\
\mathcal C(X) \arrow[dr] \arrow[rr, "\Psi"] & & \mathcal C(Y) \arrow[dl] \\
& \mathcal C(W) &
\end{tikzcd}
\caption{}
\label{fig:diamond}
\end{figure}

With Lemmas \ref{lemma:qi-curve-cx} and \ref{BoundaryCurveComplex} in hand, we are now ready to prove Theorem \ref{thm:regular}, which we restate here for convenience.

\vspace{2 mm}
\noindent \textbf{Theorem \ref{thm:regular}. }\textit{If $p: X \to S$ and $q: Y \to S$ are two regular covers so that given any closed curve $\gamma \subset S$, there exists simple elevations of $\gamma$ to $X$ if and only if there exists simple elevations of $\gamma$ to $Y$, then $p$ and $q$ are equivalent covers.}
\vspace{2 mm}

Before beginning the proof, we will give a brief overview of our strategy. First we show that the map on $\partial \mathcal{C}(S)$ induced by the quasi-isometric embeddings $q^{*}, g^{*}$ commutes with the maps on $\partial \mathcal T(S)$ induced by the isometric embeddings $\tilde{q}, \tilde{g}$ coming from Theorem \ref{thm:folklore}, with respect to the natural projection from $\mathcal{PMEL}(S)$ to $\partial \mathcal{C}(S)$. We then show that $\widetilde q$ and $\widetilde g$ agree on $\partial \mathcal T(S)$, which allows us to argue that $\widetilde{q}, \widetilde g$ agree on $\mathcal T(S)$ setwise. Finally, we leverage the description of $\mathcal T(S)$ as the space of conjugacy classes of discrete, faithful representations $[\rho]: \pi_1(S) \to \text{PSL}_{2}(\mathbb R)$ to conclude that $p$ and $q$ are equivalent covers.

\begin{proof} By Lemma \ref{BoundaryCurveComplex}, the quasi-isometric embeddings $q^{*}, g^{*}: \mathcal C(S) \to \mathcal C(Y)$ defined above induce the same map from $\partial C(S) \to \partial C(Y)$. Thus, by Theorem \ref{thm:ending-laminations}, $q^{*}$ and $g^{*}$ induce the same map from the space of ending laminations of $S$ to those of $Y$. 

We have two covering maps from $Y$ to $S$: one is the regular covering map $q: Y \to S$; the other one comes from the composition of $p: X \to S$ with the homeomorphism between $X$ and $Y$ coming from Lemma \ref{lemma:qi-curve-cx} and Theorem \ref{thm:RS-quasi}, which we name as $g$. By Theorem \ref{thm:folklore}, the covering maps induce two isometric embeddings $\widetilde{q}, \widetilde{g}: \mathcal T(S) \to \mathcal T(Y)$ which by Theorem \ref{thm:Mj} admit natural extensions to the Thurston boundary.

Consider the diagram in Figure \ref{fig:diagram}, where $\pi_S$ is the natural projection (given by forgetting the measure) from $\mathcal{PMEL}(S)$ to $\mathcal{EL}(S)$; $\pi_{Y}$ is defined similarly.

\begin{figure}
\begin{tikzcd}[column sep=huge, row sep = huge]
\mathcal{PMEL}(S) \ar[ d , " \pi_{ S } " swap ] \ar[ r , " \widetilde{ q } " ] & \mathcal{PMEL} ( Y ) \ar[ d , " \pi_{ Y } " swap ] & \mathcal{PMEL} ( S ) \ar[ l , " \widetilde{ g } " swap ] \ar[ d , " \pi_{ S } " ] \\
\partial \mathcal{ C } ( S ) \ar[ r , " q^{ * } " swap ] & \partial \mathcal{ C } ( Y ) & \partial \mathcal{ C } ( S ) \ar[ l , " g^{ * } " ]
\end{tikzcd}
\caption{}
\label{fig:diagram}
\end{figure}

\begin{proposition} \label{commuting}
The diagram in Figure~\ref{fig:diagram} commutes.
\end{proposition}

\begin{proof} Let $[ \lambda] \in \mathcal{PMEL}(S)$ be a projective measured ending lamination with $\lambda$ as a representative. Consider the pullback $q^{-1}(\lambda)$ of $\lambda$ to $Y$ under the covering map $q: Y \to S$. Note that the projective measured lamination $[q^{-1}(\lambda)]$ is precisely the image of $[\lambda]$ under $\widetilde{q}$ by Theorem \ref{thm:Mj}. Subsequently, $\pi_{Y}(\tilde{q}([\lambda])) =\pi_{Y}([q^{-1}(\lambda)])$ is the underlying geodesic lamination $q^{-1}(\lambda)_{top}$ of $[q^{-1}(\lambda)]$ obtained by forgetting the measure.

On the other hand, $\pi_S([\lambda]) = \lambda_{top}$ is the underlying geodesic lamination of $[\lambda]$. Fix an arbitrary hyperbolic metric on $S$. Since $\lambda_{top}\in \mathcal{EL}(S) = \partial \mathcal{C}(S)$, we may choose a quasi-geodesic ray $\{\alpha_i\}_{i\in \mathbb{N}}$ of simple closed geodesics on $S$ such that $\{\alpha_i\}$ converge to $\lambda_{top}$ as laminations. Consider the sequence of simple closed multigeodesics $\{q^{*}(\alpha_i)\}$ in $Y$ equipped with the hyperbolic metric induced by $q: Y \rightarrow S$. Since $q^{*}$ is a quasi-isometry that extends continuously to the boundary of $\mathcal{C}(S)$, $\{q^{*}(\alpha_i)\}$ is a quasi-geodesic ray in $\mathcal{C}(Y)$ and therefore converges to some  (minimal) geodesic lamination $q^{*}(\lambda_{top}) = z \in \partial \mathcal{C}(Y)$. We claim that this limit is exactly $q^{-1}(\lambda_{top})$. 

To see this, note that the image of $q^{*}(\alpha_i)$ under $(q^{*})^{-1}$ is  $q(q^{*}(\alpha_i))= \alpha_i$. Since $\{\alpha_i\}$ converges to $\lambda_{top}$, by continuity, $q(q^{*}(\lambda_{top})) = q(z)$ must be $\lambda_{top}$. Hence $z$ is contained in the pre-image $q^{-1}(\lambda_{top})$. Since $z$ and $q^{-1}(\lambda_{top})$ are both minimal geodesic laminations on $Y$, they must be equal. Therefore $q^{*}(\pi_S([\lambda])) = q^{-1}(\lambda_{top})$.

Finally, observe that $q^{-1}(\lambda_{top})= q^{-1}(\lambda)_{top}$ since the lifting map $q^{-1}$ commutes with the measure-forgetting map. A similar argument holds for $\widetilde{g}$ and $g^{*}$. This concludes the proof of the proposition. \end{proof}

 It follows from Theorem \ref{LM-thick} that the set of uniquely ergodic ending measured laminations $\mathcal{EL}_{ue}$ are dense in the space of measured laminations $\mathcal{ML}$. By unique ergodicity, each $\lambda \in \mathcal{EL}_{ue}$ has a unique lift to $\mathcal{PML}$ under the natural projection $\pi: \partial \mathcal{T} \to \partial \mathcal{C}$. So the preimage $\pi_S^{-1}(\mathcal{EL}_{ue}(S))$ is dense in $\mathcal{PML}(S)$. Since $q^{*}$ and $g^{*}$ induce the same map from $\mathcal{EL}(S) \to \mathcal{EL}(Y)$ and the diagram in Figure~\ref{fig:diagram} commutes, $\widetilde{q}$ and $\widetilde{g}$ agree on $\pi_S^{-1}(\mathcal{EL}_{ue}(S))$ pointwise and hence on $\mathcal{PML}(S)$ pointwise by continuity. Therefore $\widetilde{q}|_{\partial \mathcal T(S)} =  \widetilde{g}|_{\partial \mathcal T(S)}$.

Next we will show that $ \widetilde{ q } $ and $ \widetilde{ g } $ have the same image in $ \mathcal{ T } ( Y ) $. Let $y \in \widetilde{q} (\mathcal T(S)) \subset \mathcal T(Y)$, and consider its preimage $x = \widetilde{q}^{-1}(y) \in {\mathcal T}(S)$. Then $x$ lies on a bi-infinite Teichm\"uller geodesic $\gamma_{x}$ which has endpoints in $\partial \mathcal T(S)$, by Theorem \ref{MasurBoundary}. By Theorems \ref{thm:folklore} and \ref{thm:Mj}, the images $\widetilde{q}(\gamma_x)$ and $\widetilde{g}(\gamma_x)$ are also bi-infinite Teichm\"uller geodesics with well-defined endpoints in $\partial \mathcal T(Y)$. Since $\widetilde{q}$ and $\widetilde{g}$ agree pointwise on $\partial \mathcal T(S)$ and any two distinct points in the boundary of Teichm\"uller space determine a unique Teichm\"uller geodesic, we have $\widetilde{q}(\gamma_x)= \widetilde{g}(\gamma_x)$. Therefore $y \in \widetilde{g}(\gamma_x) \subset \widetilde{g}(\mathcal T(S))$, which completes the proof that $\tilde{g}$ and $\tilde{q}$ have the same image. 

It follows that $\widetilde g^{-1} \circ \widetilde q$ is an isometry of $\mathcal T(S)$, fixing $\partial \mathcal{T}(S)$. This implies  that $\widetilde g^{-1} \circ \widetilde q$ is the identity on $\mathcal T(S)$ and thus that $\widetilde q$ and $\widetilde g$ agree pointwise on $\mathcal T(S)$.

Recall that $x \in \mathcal T(S)$ can be interpreted as a conjugacy class of a discrete, faithful representation $\rho: \pi_1(S) \to \text{PSL}_2(\mathbb R)$.  Consider $\widetilde{q}([\rho]), \widetilde{g}([\rho]) \in \mathcal T(Y)$, which (up to conjugacy) are the following: 
\[\widetilde q([\rho]): \pi_1(Y) \xrightarrow{q_*} \pi_1(S) \xrightarrow{\rho} \text{PSL}(2, \mathbb R),\] 
\[ \widetilde g([\rho]): \pi_1(Y) \xrightarrow{\varphi_*} \pi_1(X) \xrightarrow{p_*} \pi_1(S) \xrightarrow{\rho} \text{PSL}(2, \mathbb R).\]
Here $\varphi_*$ is the isomorphism between $\pi_1(X)$ and $\pi_1(Y)$ induced by the  homeomorphism between $X$ and $Y$ given by Lemma \ref{lemma:qi-curve-cx}. Since $\widetilde q$ and $\widetilde g$ agree on $\mathcal T(S)$, $\widetilde q([\rho]) = \widetilde g([\rho])$. This combined with the injectivity of $\rho$ imply that $q_*(\pi_1(Y)) = p_*(\pi_1(X))$. So $p$ and $q$ are equivalent covers. \end{proof}

\section{Regular Covers and Non-simple Curves} 

We now address Theorem \ref{thm:nonsimple}. Let $k = \max(N,M)$, and given $j \in \mathbb{N}$,  define $\Gamma_{k,j}(S)$ to be the graph whose vertices correspond to isotopy classes of closed curves $\gamma$ on $S$ with  self intersection number at most $k$, and so that two such vertices are connected by an edge exactly when the corresponding curves have geometric intersection at most $j$. Given a fixed $k$, there exists $r \in \mathbb{N}$ so that for all $j \geq r$, every vertex of $\Gamma_{k,j}(S)$ is adjacent to a simple closed curve (and thus in particular, the graph is connected since $\mathcal{C}(S)$ is connected). Choose $j$ so that 
\[ j \geq r \cdot [\pi_{1}(S): \pi_{1}(X)] \cdot [\pi_{1}(S) : \pi_{1}(Y)].\]

\begin{proposition} \label{prop:totallygeo} Let $\mathcal{C}_{j}(S)$ be the graph whose vertices are isotopy classes of simple closed curves and so that two curves are connected by an edge if they intersect at most $j$ times. Then $\mathcal{C}_{j}(S)$ is totally geodesic inside $\Gamma_{k,j}(S)$ and $\Gamma_{k,j}(S)$ lies in its one-neighborhood (and hence also in the one-neighborhood of $\mathcal{C}(S)$). 
\end{proposition}

\begin{proof} The fact that $\Gamma_{k,j}(S)$ lies in the one-neighborhood of $\mathcal{C}(S)$ follows immediately from the choice of $j$. To see that $\mathcal{C}_{j}(S)$ is totally geodesic, consider the map 
\[ \tau: \Gamma_{k,j}(S) \rightarrow \mathcal{C}_{j}(S), \]
defined as follows. For each vertex $v \in \Gamma_{k,j}(S)$, choose a preferred minimal position representative in that isotopy class (which by abuse of notation we refer to as $v$), a choice of orientation, and a choice of starting point which we will denote by $i(v)$. Starting at $i(v)$, traverse $v$ in the forward direction until first reaching a point that has already been visited. This defines a sub-curve $v_{s}$ (which may not necessarily contain $i(v)$) which is necessarily simple and which is essential since $v$ was in minimal position. Note that we are using the fact that $S$ is closed here. Then define $\tau$ by $\tau(v)= v_{s}$. 

Note that $v_{1}, v_{2} \in \Gamma_{k,j}(S)$ are adjacent if and only if they intersect at most $j$ times. If this occurs, we obviously also have that $(v_{1})_{s}, (v_{2})_{s}$ intersect at most $j$ times and so it follows that $\tau$ is a $1$-Lipschitz retraction of $\Gamma_{k,j}(S)$ on to $\mathcal{C}_{j}(S)$.
\end{proof}

\begin{corollary} \label{cor:QI} The inclusion of $\mathcal{C}(S)$ into $\Gamma_{k,j}(S)$ is a quasi-isometry. 
\end{corollary}

\begin{proof} It suffices to show that $\mathcal{C}(S) \hookrightarrow \mathcal{C}_{j}(S)$ is a quasi-isometry. This follows from standard arguments; in particular, if $\alpha, \beta$ are simple closed curves intersecting at most $j$ times, their distance in $\mathcal{C}(S)$ can be at most $2\log_{2}(j)+ 2$.  
\end{proof}

\begin{proof}[Proof of Theorem~\ref{thm:nonsimple}]

We will begin by establishing a quasi-isometry $\Omega: \Gamma_{k,j}(X) \rightarrow \Gamma_{k,j}(Y)$ which is defined in precisely the same way as $\Phi$ in the proof of Theorem \ref{thm:regular}: given $\alpha \in \Gamma_{k,j}(X)$ a closed curve, first take its pre-image under the cover from $W$ to obtain the curve collection $A_{W}$ and then project to $Y$ under $q': W \rightarrow Y$.

This yields a collection of curves $\alpha_{1},..., \alpha_{n}$ on $Y$ satisfying $i(\alpha_{l}, \alpha_{l}) \leq M \leq k$, $ 1 \leq l \leq n$. The value of $j$ was chosen to ensure that $A_{W}$ has diameter $1$ in $\Gamma_{k,j}(W)$, and so it follows by the same argument used in the proof of Theorem \ref{thm:regular} that $\left\{\alpha_{1},..., \alpha_{n}\right\}$ has diameter bounded above independent of $\alpha$ in $\Gamma_{k,j}(Y)$. Just as in Theorem \ref{thm:regular}, this establishes the well-definedness of $\Omega$ and that it is coarsely Lipschitz; interchanging the roles of $Y$ and $X$ give that $\Omega$ is a quasi-isometry. 

Pre- and post-composing $\Omega$ with the quasi-isometries $\mathcal{C}(X) \hookrightarrow \Gamma_{k,j}(X)$ and $\Gamma_{k,j}(Y) \rightarrow \mathcal{C}(Y)$ from Corollary \ref{cor:QI} yields a quasi-isometry $\Phi: \mathcal{C}(X) \rightarrow \mathcal{C}(Y)$. Once $\Psi: \mathcal C(X) \to \mathcal C(Y)$ is replaced with $\Phi: \mathcal C(X) \to \mathcal C(Y)$ (just defined) in the proof of Lemma \ref{BoundaryCurveComplex}, then the remainder of the argument is identical to the proof of Theorem \ref{thm:regular}. 
\end{proof}

\section{Abelian Covers and Simple Curves}

\label{section:abelian}

Let $p: X \to S$ be a finite cover. A curve $\widetilde{\alpha} \subset X$ is an elevation of a curve $\alpha \subset S$ of degree $k$ if $\widetilde{\alpha}$ covers $\alpha$ with degree $k$.

We will consider the (integral) simple curve homology $\Hom(X) = \Hom(X; \mathbb Z)$, which is the subgroup of $H_{1}(X; \mathbb Z)$ generated by the (integral homology classes of) elevations of simple closed curves on $S$. Fixing basepoints $(X, x_0)$ and $(S, s_0)$ such that $p(x_0) = s_0$, we can find a generating set $G_p$ for $\Hom(X)$ in terms of the fundamental groups of $X$ and $S$.

Specifically, call an element of $\pi_1(S, s_0)$ simple if it is represented by a curve which is freely homotopic to a simple closed curve on $S$. We will denote the collection of all powers of simple elements by $P \subset \pi_1(S, s_0).$ Observe that \[G_p \coloneqq \{[\alpha]\in \text{H}_1(X;\mathbb Z) \mid \alpha \in \pi_1(X, x_0) \text{ and } p_*\alpha \in P\}\] is a generating set for $\Hom(X)$.

\begin{proposition}[Looijenga, \cite{Looijenga}] \label{thm:looijenga} $\Hom(X) = \text{H}_1(X)$ for a finite-degree abelian cover $p: X \to S$.\end{proposition}

Given two finite-degree abelian covers $p: X \to S$ and $q: Y \to S$ and a curve $\gamma \subset S$, we will let $n_p(\gamma)$ (resp. $n_q(\gamma)$) denote the minimum positive integer $k$ such that $\gamma$ admits an elevation of degree $k$ along $p$ to $X$ (resp. along $q$ to $Y$). 

\vspace{2 mm}

\noindent \textbf{Theorem \ref{thm:abelian}. }\textit{Let $p: X \to S$ and $q: Y \to S$ be finite-degree abelian covers of $S$. If $n_p(\gamma) = n_q(\gamma)$ for all simple closed curves $\gamma \subset S$, then $p$ and $q$ are equivalent covers.}

\vspace{2 mm}

\begin{proof}
Fix basepoints $x_0, y_0$, and $s_0$ of $X, Y$, and $S$, respectively, such that $p(x_0) = q(y_0) = s_0$. By Proposition~\ref{thm:looijenga}, $G_p$ and $G_q$ are generating sets for $\Hom(X) = \text{H}_1(X)$ and $\Hom(Y) = \text{H}_1(Y)$, respectively, and so $p_*G_p$ and $q_*G_q$ are generating sets for $p_*\text{H}_1(X)$ and $q_*\text{H}_1(Y)$. Moreover, we have that
\begin{align*} 
p_*G_p &= \{p_*[\alpha] \in \text{H}_1(S) \mid \alpha \in \pi_1(X,x_0) \text{ and } p_*\alpha \in P\} \\
&= \{[\gamma] \in \text{H}_1(S) \mid \gamma \in P \cap p_* \pi_1(X, x_0)\} \\
&= \{[\gamma] \in \text{H}_1(S) \mid \gamma \in P \cap p_* \pi_1(Y, y_0)\} \\
&= \{q_*[\beta] \in \text{H}_1(S) \mid \beta \in \pi_1(Y, y_0) \text{ and } q_* \beta \in P\} = q_* G_q, \\
\end{align*}
and so $p_*\text{H}_1(X) = q_* \text{H}_1(Y).$ Since $p$ and $q$ are abelian covers, they must be equivalent.
\end{proof}

\bibliographystyle{plain}
\bibliography{main}

\end{document}